\documentclass{amsart}
\date{May 16, 2017}

\usepackage{amssymb,graphicx,amssymb,amsmath,lineno,amsopn}
\usepackage[ansinew]{inputenc}
\usepackage[T1]{fontenc}

\usepackage{textcomp}
\usepackage{charter}
\usepackage[]{bera}
\usepackage[charter]{mathdesign}

\usepackage{ifpdf}
\ifpdf
\usepackage[%
  pdftitle={Instructions for use of the document classelsart},%
  pdfauthor={Simon Pepping},%
  pdfsubject={The preprint document class elsart},%
  pdfkeywords={instructions for use, elsart, document class},%
  pdfstartview=FitH,%
  bookmarks=true,%
  bookmarksopen=true,%
  breaklinks=true,%
  colorlinks=true,%
  linkcolor=blue,anchorcolor=blue,%
  citecolor=blue,filecolor=blue,%
  menucolor=blue,pagecolor=blue,%
  urlcolor=blue]{hyperref}
\else
\usepackage[%
  breaklinks=true,%
  colorlinks=true,%
  linkcolor=blue,anchorcolor=blue,%
  citecolor=blue,filecolor=blue,%
  menucolor=blue,pagecolor=blue,%
  urlcolor=blue]{hyperref}
\fi

\makeatletter
\def\elsartstyle{%
    \def\normalsize{\@setfontsize\normalsize\@xiipt{14.5}}
    \def\small{\@setfontsize\small\@xipt{13.6}}
    \let\footnotesize=\small
    \def\large{\@setfontsize\large\@xivpt{18}}
    \def\Large{\@setfontsize\Large\@xviipt{22}}
    \skip\@mpfootins = 18\p@ \@plus 2\p@
    \normalsize
} \@ifundefined{square}{}{} \makeatother

\newtheorem{teo}{Theorem}

\newtheorem{lema}{Lemma}
\theoremstyle{definition}
\newtheorem{defi}{Definition}
\theoremstyle{remark}

\newtheorem{example}{Example}

\newcommand{\brn}{\begin{eqnarray*}}
\newcommand{\ern}{\end{eqnarray*}}

\title[Ratio asymptotic for bi-orthogonal matrix polynomials]{Ratio asymptotic for bi-orthogonal matrix polynomials with unbounded recurrence coefficients.
} 

\begin{document}
\author[A. Branquinho]{Amilcar  Branquinho}
\address{CMUC 	and Department of Ma\-the\-ma\-tics, University of Coimbra, Apartado 3008, EC Santa Cruz, 3001-501 COIMBRA, Portugal.}
\email{ajplb@mat.uc.pt}
\thanks{AB acknowledges Centro de Matem\'{a}tica da Universidade de Coimbra (CMUC) -- UID/MAT/00324/2013, funded by the Portuguese Government through FCT/MEC and co-funded by the European Regional Development Fund through the Partnership Agreement PT2020.}

\author[J.C. García-Ardila]{Juan Carlos García-Ardila}
\address{Departamento de Matemáticas, Universidad Carlos III de Madrid, Avenida Universidad 30, 28911 Leganés, Spain}
\email{jugarcia@math.uc3m.es }

\author[F. Marcellán]{Francisco Marcellán}
\address{Departamento de Matemáticas, Universidad Carlos III de Madrid and Instituto de Ciencias Matemáticas (ICMAT), Avenida Universidad 30, 28911 Leganés, Spain}
\email{pacomarc@ing.uc3m.es }
\thanks{MM \& FM thanks financial support from the Spanish ``Ministerio de Economía y Competitividad" research project MTM2012-36732-C03-01,  \emph{Ortogonalidad y aproximación; teoría y aplicaciones}}


 \begin{abstract}
In this work is presented a study on matrix biorthogonal polynomials sequences that satisfy a nonsymmetric recurrence relation with unbounded coefficients. The ratio asymptotic for this family of matrix biorthogonal polynomials is derived in quite general assumptions. It is considered some illustrative examples.
 \end{abstract}

 \maketitle

\section{Introduction} \label{sec:1}

The study of the outer ratio asymptotics i.e. the limit of the ratio of two consecutive polynomials $p_{n} $ and $p_{n+1} $ of a sequence of polynomials,~$\{ p_n \}_{n \in \mathbb N}$, orthogonal with respect to an inner product outside the convex hull of the support of the measure of orthogonality 
has attracted
the interest of many researchers in the last decades. Since the first study 
of 
Nevai in 1979 (cf.~\cite{N1}) for orthogonal polynomials with respect to a measure supported on a infinite subset of the real line   with convergent recurrence coefficients, more general situations have been considered, such as the case of asymptotically periodic recurrence coefficients with a finite number of accumulation points (cf.~\cite{GW1},~\cite{W2},~\cite{W1}) or the case of unbounded recurrence coefficients (cf.~\cite{W2}).

In 1993  
Durán (cf.~\cite{Duran3}) gave the characterization of symmetric bilinear forms for which the multiplication operator by a polynomial is a symmetric one.
There, necessary and sufficient conditions were deduced, so that, a sequence of scalar polynomials ~$\{ p_n \}_{n \in \mathbb N} $ satisfying a $(2N+1)$-term recurrence~relation
 \begin{gather*}
 h(x) \, p_n(x)=c_{n,0} \, p_n(x)+\sum_{k=1}^{N} \big( \overline{c}_{n,k} \, p_{n-k}(x) +c_{n+k,k} \, p_{n+k}(x) \big) \, ,
 \end{gather*}
 is orthogonal with respect to a symmetric bilinear form (generalization of the Favard's theorem). In particular, their attention was focused on the discrete Sobolev type inner products. In that work, the author gave  a first idea to connect  
 scalar orthogonal  polynomials with respect to a bilinear form and matrix orthogonal polynomials with respect to a positive definite matrix  of measures. From the above result, 
Durán and 
Van Assche~\cite{DW1} proved  that if~$
\{ p_n \}_{n \in \mathbb N}
 $
is a sequence of scalar polynomials satisfying a $(2N+1)$-term recurrence relation, they are  related to a matrix polynomial sequence
$
\{ P_n \}_{n \in \mathbb N}
$
satisfying a matrix three-term recurrence~relation
\begin{gather}\label{11}
x \, P_n(x)=A_{n+1} \, P_{n+1}(x)+B_{n} \, P_n(x)+ {A}^{*}_n \, P_{n-1}(x) \, ,\ \ n \in \mathbb N \, ,
\end{gather}
with initial conditions  $P_{0}(x)=I_{N
}$ and $P_{-1}(x)=\pmb 0_{N
} \, $,
where for each $n \in \mathbb N$, $A_{n}$ is an upper triangular, nonsingular matrix and $B_{n}$ is a Hermitian matrix.

The above results reawakened the interest on matrix orthogonal polynomials (cf. the survey paper~\cite{DPS}). So, in~\cite{D} 
Durán and 
López-Rodrí\-guez studied properties for the zeros of a sequence of matrix polynomials~$\{ P_n \}_{n \in \mathbb N}
$
which are orthonormal  with respect to a positive definite matrix of measures~$W$. Next,  
Durán in~\cite{Duran4} showed two important results: the first one is a  quadrature formula for matrix polynomials and the second one is the Markov theorem for matrix polynomials when again the matrix of measures is positive definite. In~\cite{Duran1}, the author  deals with  the outer ratio asymptotic for matrix orthogonal polynomials. Therein, 
Durán obtained the asymptotic behavior of two consecutive polynomials belonging to the matrix Nevai class, i.e. these polynomials satisfy a  three-term recurrence formula as in~\eqref{11} where, again, $A_{n}$ are nonsingular, upper triangular matrices and $B_{n}$ are Hermitian matrices for all $n \in \mathbb N$, and  such that $A_n \to A$ and $B_n \to B$.
 Later on, 
 Durán and 
 Daneri-Vias analyzed the  above case but when  the matrix sequences
 $
 (A_n)_{n \in \mathbb N}
 $,
 $
(B_n)_{n\in \mathbb N}
 $
diverge in a particular way
~(cf.~\cite{Duran2} for details).

Recently, 
Yakhlef and  
Marcellán~\cite{mar}  have studied the outer relative asymptotics of sequences of matrix orthogonal polynomials for Uvarov perturbations in  the degenerate case,~i.e. given a positive definite matrix of measures $\alpha$, and its corresponding sequences of matrix orthonormal polynomials,
 $
\{ P_n^{\alpha} \}_{n\in \mathbb N}
 $,
satisfying a three-term recurrence relation as in~\eqref{11}, they
define a new matrix of measures~$\beta$~as
\begin{gather*}
d\beta(u)= d\alpha(u)+M\delta(u - c) \, ,
\end{gather*}
where $M$ is a positive definite matrix, $\delta(u-c)$ is the Dirac measure  supported at $c$ that is located outside the support of $d\alpha$, and the sequence of matrix orthonormal polynomial associated with $d\beta$, $ \{ P_n^\beta \}_{n \in \mathbb N} $. Then,
they study the outer relative asymptotic 
between the sequences
$\{ P_n^{\beta}\}_{n\in \mathbb N}
$
and
$\{ P_n^{\alpha} \}_{n\in \mathbb N}
$
under quite general assumptions on  the coefficients of the three-term recurrence relation
$
(A_n)_{n\in \mathbb N}
$,
$
(B_n)_{n\in \mathbb N}
$.

When the matrix of measures, $W$, is no longer Hermitian, we can define a bilinear (respectively, sesquilinear form) in $\mathbb R^{N \times N}[x] $ (respectively, in $\mathbb C^{N \times N}[x] $) and deal with sequences of  biorthogonal matrix polynomials,
$\{ V_n \}_{n\in \mathbb N} \,
$,
$\{ G_n \}_{n\in \mathbb N}
$,
which play the role of the left and right-orthogonalities (cf. Definition~\ref{defi1}).
It can be proven that these sequences satisfies a three-term recurrence relation 
\begin{gather}  \label{re}
x \, V_n(x) =\alpha_{n} \, V_{n+1}(x)+\beta_{n} \, V_n(x)+\gamma_n \, V_{n-1}(x) \, , \\
 x \, G_n(x) =a_{n} \, G_{n+1}(x)+b_{n} \, G_n(x) + c_n \, G_{n-1}(x) \, ,
 \label{re1}
\end{gather}
where $ a_n $,  $\alpha_{n} $,  $b_{n} $, $\beta_{n} $, are nonsingular matrices.  Without loss of generality we can suppose that $ a_n $, $\alpha_{n}$ are  lower triangular matrices and $c_{n} $, $\gamma_{n}$ are upper triangular matrices (cf.~\cite{DW1}).

As we  have said above, matrix polynomials defined  by the recurrence formulas as~\eqref{re} appear in a natural way in the literature and its study paid an increasing attention in the last decades.  For example in~\cite{Bra-Cot-Fou}, the authors give a matrix interpretation of the multiple orthogonality in terms of matrix orthogonal polynomials satisfying the same kind of recursion formulas. On the other hand, in  ~\cite{ACM},~\cite{ACM1},~\cite{gar1} were studied perturbations of measures (Christoffel, Geronimus, and Geronimus-Uvarov) which yield to non-positive definite matrix of measuress, and thus, to the biorthogonality.

\begin{example}\label{exa1}
Let $\mu$ be a scalar measure supported on the real line, and
$ \{ p_n \}_{n\in \mathbb N}
$
its corresponding sequence of monic orthogonal polynomials. If $W $ is a matrix polynomial of degree $M$, then we can define a new measure  $d \hat\mu=W \, d\mu$ (Christoffel transformation of the measure) which clearly is non-positive definite as~$W$ needs not to be identical to
${W}^{\mathsf T}$. In particular,~taking
$W(x)=
\left(
\begin{smallmatrix}
x&1\\ 0 &x
\end{smallmatrix}
\right)
\, $,
then following the techniques developed in~\cite{ACM}, it is easy to see that the sequence of   matrix polynomials
\begin{gather*}
 V_n(x)
=
\left( \begin{matrix}
 {K_{n} (x,0)}/{p_{n}(0)} &
 \frac{p_n (0) \, K_n (x,0) - W(p_n,p_{n+1}) (0) \, p_n (x) }{p_n^2(0) \, x}
\\ 
 0 &
 {K_{n}(x,0)}/{p_n(0)}
\end{matrix}
\right) \, , \\
 G_n(x) =-
\begin{pmatrix}
{K_n(x,0)}/{p_n(0)} & 0 \\
-
\big( {K_n(x,0)p_n'(0)}+
\frac{\partial K_n(x,0)}{\partial y} \big)/p_n(0) &
{K_n(x,0)}/{p_n(0)}
\end{pmatrix}
\, . 
\end{gather*}
with
$W(p_n,p_{n+1})(x)=p_n(x) p'_{n+1}(x)-p_{n+1}(x)p'_{n}(x) \, $,
and
\begin{gather*}
K_n(x,0) = 
\big( {p_n(0)} \, {p_{n+1}(x)- 
{p_{n+1}(0)} \, p_n(x)}\big) / {x}
\end{gather*}
satisfies 
\begin{gather*}
\int  V_m(x) \, W(x) \,  G_n(x) \, d \mu (x) = I_N \, \delta_{n,m} \, ,
\end{gather*}
i.e.
$\{ V_n \}_{n\in \mathbb N}
$,
$\{ G_n \}_{n\in \mathbb N}
$
are sequences of biorthonormal polynomial with respect to~$W$ and
satisfy the 
three-term recurrence relations
\eqref{re} and~\eqref{re1} with $a_n = \gamma_{n+1}$, $b_n = \beta_n$, and $c_n = \alpha_{n-1} = I_N$ (cf. Theorem~\ref{t:BMF}) .
\end{example}
Outer ratio asymptotics for the class of matrix orthogonal polynomials satisfying the three-term recurrence formula as in~\eqref{re} was studied for the first time in~\cite{BMF1}. There, the authors analyzed the case of convergent recurrence coefficients by introducing an analog of the  Nevai class of matrix polynomials for the nonsymmetric case, the generalized matrix Nevai class.

In the present contribution our aim is to generalize those results when the coefficients of  the nonsymmetric recurrence formula diverge in a particular~way.


The structure of the manuscript is as follows. Section~\ref{sec:2} provides the basic background about  matrix biorthogonal polynomials $\{V_n \}_{n \in \mathbb N}$ and $\{G_n \}_{n \in \mathbb N}$ that satisfy dual nonsymmetric recurrence relations. Here, we establish the relations 
between the zeros of these two families of polynomials and discuss the most appropriate way of scaling these matrix polynomials in order to obtain its asymptotic behavior. Section~\ref{sec:3} deals with the outer ratio asymptotics for left and right-orthogonal matrix polynomials with varying recurrence coefficients. We also deduce a quadrature 
and a Liouville-Ostrogradski formulas for the right-orthogonal polynomials. Section~\ref{sec:4} is focused on our main result (Theorem
\ref{main result}), the outer ratio asymptotics of matrix orthogonal polynomials satisfying recurrence formulas with nonsymmetric and nonsingular recurrence coefficients diverging in a particular way. Section~\ref{sec:5} is devoted to the study of the case when certain matrix appearing in the recurrence formula is singular.



\section{Matrix biorthogonal polynomials } \label{sec:2}

Let us a consider a {\it quasidefinite $N\times N$ matrix of measures}, $W$, i.e.
$W = \big( w_{i,j} \big)_{i,j=0}^{N-1}$, with measures, $w_{i,j}$, $ i,j \in \{ 0 , \ldots ,N-1\} $, supported on the real line but not necessarily positive definite with finite moments, $U_n = \int x^n \, d W(x) \, $, $n \in \mathbb N$, and such that the Hankel determinants satisfy
\begin{gather*}
\det \big( \left( U_{i+j} \right)_{i=0,\ldots , n}^{j=0,\ldots , n} \big) \not = 0 \, , \ \ \ n \in \mathbb N \, .
\end{gather*}
 We will assume, without loss of generality, that the matrix of measures is normalized by $U_0 = \int d W(x) =I_N$.
If $P $ and $R $ are matrix polynomials in $\mathbb{C}^{N\times N}[x]$, then we introduce the following {\it sesquilinear form},
\begin{gather*}
\langle P \, , \, R \rangle = \int P(x) \, dW (x) \, R^{\mathsf T}(x) \, , \  P \, , \ R \, \in \mathbb C^{N\times N}[x] \, .
\end{gather*}
\begin{defi}\label{defi1}
Let $W$ be a quasidefinite $N\times N$ matrix of measures.
The matrix polynomial sequences,
$\{ V_n \}_{n\in \mathbb N} $ (respectively, $\{ G_n \}_{n\in \mathbb N} $),
such that for every $n, m\in \mathbb N$, $\deg V_n(x)= n $
(respectively, $\deg G_m(x)=m$) and
\begin{gather*}
\int V_n(x) \, dW (x) \, x^m =\Omega^{(1)}_{n} \, \delta_{n,m}\, , \ \  m = 0, \ldots , n \\
\mbox{(respectively, } \ \ \int  x^n \, dW (x) \, G_m(x)=\Omega^{(2)}_{m} \, \delta_{n,m}\, , \ \  n = 0, \ldots , m \, \mbox{)},
\end{gather*}
with $\delta_{n,m}$ is the Kronecker symbol and
$\Omega^{(1)}_{n}$ (respectively, $\Omega^{(2)}_{n}$)
nonsingular matrices for $n \in \mathbb N$,
are said to be {\it the left} (respectively, {\it right}) {\it orthogonal polynomial sequences} with respect to $W$.
\end{defi}
We also refer to
$\{ V_n \}_{n\in \mathbb N}
$,
$\{ G_n \}_{n\in \mathbb N}
$
as 
biorthogonal polynomial sequences with respect to the quasidefinite matrix of measures~$W$, when
\begin{gather*}
\int V_n(x) \, dW (x) \, G_m(x) = I_N \, \delta_{n,m} \, , \ \ n , m \in \mathbb N \, .
\end{gather*}
\begin{teo}[cf. \cite{BMF2}, Theorem 4]\label{t:BMF}
Given a quasidefinite matrix of measures~$W$, then its biorthogonal polynomial sequences,
$ \{ V_n \}_{n\in \mathbb N}
$,~$ \{ G_n \}_{n\in \mathbb N}
$
 satisfy the 
three-term recurrence~relations
\begin{gather}\label{5trr_Vn}
x \, V_n(x)=A_n \, V_{n+1}(x)+B_n \, V_n(x)+C_n \, V_{n-1}(x) \, , \ \   n \in \mathbb N \, , \\
\label{5trr_Gn}
x \, G_n(x)
=
G_{n-1}(x) \, A_{n-1}+ G_n(x) \, B_n+ G_{n+1}(x) \, C_{n+1} \, , \ \  n \in \mathbb N \, ,
\end{gather}
with $V_{-1}(x)=\pmb 0_{N
}$, $V_{0}(x)=I_{N
}$, and
 $G_{-1}(x)=\pmb 0_{N
 }$, $G_{0}(x)=I_{N
 } $.
Here, $A_n$ and $C_{n} $  are nonsingular matrices for every $n\in \mathbb N$.  \end{teo}
Without loss of generality we can suppose that $ (A_n )_{n \in \mathbb N} $ (respectively, $(C_{n})_{n \in \mathbb N}$) is a sequence of lower (respectively, upper) triangular matrices (cf. \cite{DW1}).

In the same way as in the scalar case, the Favard's Theorem for matrix polynomials can be find in the literature (cf. \cite{BMF2}, Theorem 7).
\begin{teo}\label{zeros}
Given a quasidefinite matrix of measures~$W$, then its biorthogonal polynomial sequences,
$ \{ V_n \}_{n\in \mathbb N}
$,~$ \{ G_n \}_{n\in \mathbb N}
$,
are such that
for each $n$, $V_n $ and $G_n $ have the same zeros.
\end{teo}
\begin{proof}
Notice that the $N$-block Jacobi~matrix associated with  the recurrence relation~\eqref{5trr_Gn} for the polynomials $G_n $ is the transpose of the $N$-block Jacobi~matrix
\begin{gather*}
J=\left(\begin{matrix}
B_0 & A_0 & \pmb 0_{N
} &   \\
C_1 & B_1 & A_1 & \ddots \\
\pmb 0_{N
}& C_2&B_2 & \ddots \\
& \ddots & \ddots & \ddots
\end{matrix}\right) \, ,
\end{gather*}
associated with the recurrence relation~\eqref{5trr_Vn} for the polynomials $V_n \, $.
The result follows by taking into account that the zeros of $V_n $ ($G_n $, respectively) are the eigenvalues of $J_n$ ($J_n^{\mathsf T}$, respectively), where $J_n$ is the truncated matrix of $J$, with dimension $nN \times nN$.
\end{proof}
\begin{defi}
Let $W$ be a quasidefinite matrix of measures and
$\{ V_n \}_{n\in \mathbb N}
$
and
$\{ G_n \}_{n\in \mathbb N}
$
the corresponding sequences of biorthogonal matrix polynomials.
We define the first kind associated polynomial sequence
$ \{ V^{(1)}_n \}_{n\in \mathbb N}
$ and
$\{ G^{(1)}_n \}_{n\in \mathbb N}
$,
as follows,
\begin{gather*}
V^{(1)}_{n-1}(x)=\int\cfrac{V_{n}(x)-V_{n}(y)}{x-y} \, dW(y) \, , \ G^{(1)}_{n-1}(x)=\int dW(y) \, \cfrac{G_{n}(x)-G_{n}(y)}{x-y} .
\end{gather*}
\end{defi}
The first kind associated polynomial sequences
$\{ V^{(1)}_n \}_{n\in \mathbb N}
$,
$\{G^{(1)}_n \}_{n\in \mathbb N}
$
also satisfy the three-term recurrence relations~\eqref{5trr_Vn} and~\eqref{5trr_Gn} 
with initial conditions,
$V_{-1}^{(1)}(x) = \pmb 0_{N
}$,
$V_{0}^{(1)}(x)=A_0^{-1}$, and 
$G_{-1}^{(1)}(x)=\pmb 0_{N
}$, $G_{0}^{(1)}(x)=C_1^{-1} \, $.

In order to obtain the outer ratio asymptotic for polynomials with varying recurrence coefficients we will need some auxiliary results such as quadrature 
and Liouville-Ostrogradsky type formulas for biorthogonal polynomials.  For the left-orthogonal  polynomials these results can be found in the literature.

\begin{lema}[cf. \cite{D}, Lemma 2.2]
Let $A $ be an $N\times N$ matrix polynomials and let~${\bf a}$ be a zero of $A$ of multiplicity~p, i.e.~${\bf a}$ is a zero of $\det A $ of multiplicity~$p$. We~put
\ $L({\bf a} \, , \, A)= \{v \in \mathbb C^N : \overline{v}^{\mathsf T} \, A \, ({\bf a}) =\pmb 0_{N
}\} $ and  $ R({\bf a} \, , \, A)=\{v \in \mathbb C^N : A \, ({\bf a}) \, v=\pmb 0_{N
}\} \, $.
If \ $\operatorname{dim} \, 
L({\bf a} \, , \, A) 
 =\operatorname{dim} \, 
R({\bf a} \, , \, A) 
=p$, \ then \ \ $\big( \operatorname{Adj} \, (A(x))\big)^{(j)}({\bf a} )=\pmb 0_{N
}$, for $j=0,\ldots, p-2$, \
and \ \
$ \big( \operatorname{Adj} \, (A(x))\big)^{(p-1)}({\bf a} )\neq \pmb 0_{N
}$. \\
Moreover, $\operatorname{rank} \, \big(\operatorname{Adj} \, (A(x))\big)^{(p-1)}({\bf a})=p \, $.
\end{lema}

\begin{lema}[cf. \cite{Berg}, Proposition~5.14]
\label{decomposition}
Let $P_n 
$ be a matrix polynomial of degree~$n$ with $m$ different zeros $\{x_{n,1},\ldots, x_{n,m}\}$ and with $\{ \ell_ 1,\ldots,\ell_ m \}$ 
as corresponding multiplicities. For any matrix polynomial $R 
$ of degree less than or equal to $n-1$ and $x \in \mathbb C \setminus \{x_{n,1},\ldots, x_{n,m}\}$ we have
\begin{gather*}
R(x) \, (P_n(x))^{-1}=\sum_{k=1}^m 
\frac{C_{n,k}}{x-x_{n,k}} \, , \ \ \
(P_n(x))^{-1}R(x) =\sum_{k=1}^m 
\frac{D_{n,k}}{x-x_{n,k}} \, ,
\end{gather*}
where
\begin{gather*}
C_{n,k} = 
\frac{\ell_ k}{(\det P_n)^{(\ell_ k)}(x_{n,k})}R(x_{n,k}) \,\big( \operatorname{Adj} \,  P_n(x) \big)^{(\ell_ k-1)}(x_{n,k}) \, , \\
D_{n,k}  = 
\frac{\ell_ k}{(\det P_n)^{(\ell_ k)}(x_{n,k})} \big( \operatorname{Adj} \,  P_n(x) \big)^{(\ell_ k-1)}(x_{n,k}) \, R(x_{n,k}) \, .
\end{gather*}
\end{lema}


\begin{teo}[Quadrature formula]
\label{quadrature_formula_right}
Let $\{ V_n \}_{n\in \mathbb N}
$,  $\{ G_n \}_{n\in\mathbb{N}}
$
be the sequences of biorthogonal matrix polynomials with
respect to a quasidefinite matrix of measures~$W$, and let
$\{ V_n^{(1)} \}_{n\in\mathbb{N}}
$,
$\{ G_n^{(1)}\}_{n\in\mathbb{N}}
$
be its first kind associated polynomial sequences.
Given the different zeros of $V_n$, $\{x_{n,1},\ldots, x_{n,s}\} $, with 
multiplicities $\{\ell_ 1, \ldots, \ell_ s\}$, 
we define the matrices
$\Gamma_{m,k}$,
$\widetilde \Gamma_{m,k}$, 
as
\begin{gather*}
\Gamma_{n,k} =  \frac{\ell_ k}{(\det V_n)^{(\ell_ k)}(x_{n,k})} \big( \operatorname{Adj} \, V_n(x) \big)^{(\ell_ k-1)}(x_{n,k})V_{n-1}^{(1)} (x_{n,k}) \, , 
\end{gather*}
\begin{gather*}
\widetilde \Gamma_{n,k} = G_{n-1}^{(1)} (x_{n,k}) \frac{\ell_ k}{(\det G_n)^{(\ell_ k)}(x_{n,k})} \big(\operatorname{Adj} \, G_n(x)\big)^{(\ell_ k-1)}(x_{n,k}) \, .
\end{gather*}
Then, for any polynomial $P $ of degree less than or equal to $2n-1$ the following quadrature formula holds
\begin{gather*}
\int P(x) \, dW(x)  =  \sum_{k=1}^{s} P (x_{n,k}) \, \Gamma_{n,k} \, ,
 \ \ \
\int dW(x) \, P(x) =  \sum_{k=1}^{s} \widetilde \Gamma_{n,k} \, P(x_{n,k}) \, .
\end{gather*}
\end{teo}
\begin{proof}
We will prove the quadrature formula for the right orthogonal polynomials, because the left one is already proved in~\cite{BMF1}.
Let $P $ be a matrix polynomial of degree less than or equal to $n-1$. Since $G_n $ is a polynomial with nonsingular leading coefficient, then
$P(x)= G_n (x) \, C(x)+R(x) \, $.
Here $C(x)$ and $R(x)$ are matrix polynomials with degree of $R(x)$ less than or equal to $2n-1$.
Using Lemma~\ref{decomposition} we get
\begin{gather*}
G_n^{-1}(x) \, R(x)=  \sum_{k=1}^{s} \frac{D_{n,k}}{x-x_{n,k}} \, ,
\end{gather*}
where the matrices $D_{n,k}$ are 
\begin{gather*}
D_{n,k} = \frac{\ell_ k}{(\det G_n)^{(\ell_ k)} (x_{n,k})} \, \big(\operatorname{Adj} \, G_n(x)\big)^{(\ell_ k-1)} (x_{n,k}) \, R(x_{n,k}) \, .
\end{gather*}
Taking into account that $R(x_{n,k})=P(x_{n,k})-G_n(x_{n,k}) \, C(x_{n,k})$ and
\begin{gather}\label{anulacion}
\big(\operatorname{Adj} \, G_n \big)^{(\ell_ k-1)} (x_{n,k}) \, G_n(x_{n,k})=G_n(x_{n,k}) \, \big(\operatorname{Adj} \,  G_n \big)^{(\ell_ k-1)} (x_{n,k}) =\pmb 0_{N
} \, ,
\end{gather}
the previous expression becomes
\begin{gather*}
D_{n,k} = \frac{\ell_ k}{(\det G_n)^{(\ell_ k)} (x_{n,k})}  \big(\operatorname{Adj} \,  G_n(x)\big)^{(\ell_ k-1)} (x_{n,k}) \, P(x_{n,k}) \, .
\end{gather*}
Thus,
\begin{gather*}
P(x)=G_n(x) \, C(x)+  \sum_{k=1}^{s} G_n(x) \, \frac{D_{n,k}}{x-x_{n,k}} \, .
\end{gather*}
Using again~\eqref{anulacion},  we have
\begin{gather*}
P(x)=G_n(x) \, C(x)+  \sum_{k=1}^{s} \frac{G_n(x)-G_n(x_{n,k})}{x-x_{n,k}} \, D_{n,k} \, .
\end{gather*}
and by the definition of the first kind associated polynomial, it follows~that
\begin{gather*}
\int dW(x) \, P(x)=\int dW(x) \, G_n(x) \, C(x) +  \sum_{k=1}^{s} G_{n-1}^{(1)}(x_{n,k}) \, D_{n,k} \, .
\end{gather*}
So, from orthogonality we have
\begin{gather*}
\int dW(x) \, P(x)=  \sum_{k=1}^{s} G_{m-1}^{(1)}(x_{m,k}) \, D_{m,k} \, ,
\end{gather*}
and the result follows.
\end{proof}
\begin{teo}[Liouville-Ostrogradski 
formula
]
Let
$\{ V_n \}_{n\in\mathbb{N}}
$,
$\{ G_n \}_{n\in\mathbb{N}}
$
be the sequence of matrix  biorthogonal polynomials with respect to a quasidefinite matrix of measures~$W $ and
$ \{ V^{(1)}_n \}_{n\in\mathbb{N}}
$,
$\{ G_n^{(1)}\}_{n\in\mathbb{N}}
$
be, respectively, the first kind associated polynomial sequences. Then, 
\begin{gather}
 \label{Liouville-Ostrogradski1}
V_n(z) \, G_n^{(1)}(z)- V_{n+1}^{(1)}(z) \, G_{n+1}(z)=A_{n}^{-1} \, , \\
\label{Liouville-Ostrogradski}
V_n(x) \, G_n^{(1)}(x)- V_{n-1}^{(1)}(x) \, G_{n+1}(x)=C_{n+1}^{-1} \, ,
\end{gather}
where $A_n$, $C_n$ are the nonsingular matrices 
in~\eqref{5trr_Gn}.
\end{teo}
\begin{proof}
Equation~\eqref{Liouville-Ostrogradski1} was already proved in~\cite{BMF1}.
To prove~\eqref{Liouville-Ostrogradski} 
we proceed by induction on $n$. For $n=0$ the result follows from the initial conditions.
We assume that the formula
\begin{gather*}
V_p(x) \, G_p^{(1)}(x)- V_{p-1}^{(1)}(x) \, G_{p+1}(x)=C_{p+1}^{-1} \, ,
\end{gather*}
is true for $p=1,\ldots,n-1$.
First, we use the recurrence relation in $G_n^{(1)}$ and~$G_{n+1}$ to obtain
\begin{multline*}
V_n(x) \, G_n^{(1)}(x)- V_{n-1}^{(1)}(x) \, G_{n+1}(x)=
 (V_n(x)  \, G_{n-1}^{(1)}(x)- V_{n-1}^{(1)}(x) \, G_n(x)) \\
 \times (xI_{N
}-B_n) \, C_{n+1}^{-1}
 -(V_n(x) \, G_{n-2}^{(1)}(x)- V_{n-1}^{(1)}(x) \, G_{n-1}(x) ) \, A_{n-1} \, C_{n+1}^{-1}
 \, .
\end{multline*}
Second, we prove that
\begin{gather*}
V_n(x) \, G_{n-1}^{(1)}(x)- V_{n-1}^{(1)}(x) \, G_n(x)=\pmb 0_{N
} \, .
\end{gather*} 	
Using the definition of the first kind associated polynomial, we get
\begin{multline*}
V_n(x) \, G_{n-1}^{(1)}(x)- V_{n-1}^{(1)}(x) \, G_n(x)
 \\ =\int\frac{V_n(y) \, dW(y)}{x-y} \, G_n(x) - V_n(x) \, \int\frac{dW(y) \, G_n(y)}{x-y} \, .
\end{multline*}
Adding and subtracting
$\displaystyle \int \frac{V_{n}(y) \, dW(y) \, G_n(y)}{x-y} \, $
in the last relation and taking in account the left and right-orthogonalities, the result follows.
With this in~mind
\begin{multline}\label{eq10}
V_n(x) \, G_n^{(1)}(x)- V_{n-1}^{(1)}(x) \, G_{n+1}(x)
	\\
=-(V_n(x) \, G_{n-2}^{(1)}(x)- V_{n-1}^{(1)}(x) \, G_{n-1}(x) ) \, A_{n-1} \, C_{n+1}^{-1} \, .
\end{multline}
Now, using the recurrence relations for $V_n $ and $V_{n-1}^{(1)} $,
\begin{multline*}
V_n(x) \, G_{n-2}^{(1)}(x)- V_{n-1}^{(1)}(x) \, G_{n-1}(x)
 \\ = A_{n-1}^{-1}(x-B_{n-1}) \, (V_{n-1}(x) \, G_{n-2}^{(1)}(x) - V_{n-2}^{(1)}(x) \, G_{n-1}(x))
 \\ - A_{n-1}^{-1} C_{n-1} \, (V_{n-2}(x) \, G_{n-2}^{(1)}(x) - V_{n-3}^{(1)}(x) \, G_{n-1}(x) ) \, .
\end{multline*}
Since $V_{n-1}(x) \, G_{n-2}^{(1)}(x)- V_{n-2}^{(1)}(x) \, G_{n-1}(x)=\pmb 0_{p
} $, we deduce
\begin{multline*}
V_n(x) \, G_{n-2}^{(1)}(x)- V_{n-1}^{(1)}(x) \, G_{n-1}(x)
\\ = - A_{n-1}^{-1} \, C_{n-1} (V_{n-2}(x) \, G_{n-2}^{(1)}(x)- V_{n-3}^{(1)}(x) \, G_{n-1}(x) ) \, .
\end{multline*}
Using this relation in~\eqref{eq10} we obtain
\begin{multline*}
V_n(x) \, G_n^{(1)}(x)- V_{n-1}^{(1)}(x) \, G_{n+1}(x)
 \\ = A_{n-1}^{-1} \, C_{n-1} \, (V_{n-2}(x) \, G_{n-2}^{(1)}(x)- V_{n-3}^{(1)}(x) \, G_{n-1}(x) ) \, A_{n-1} \, C_{n+1}^{-1} \, .
\end{multline*}
According to the induction hypothesis the result follows.
\end{proof}
In the sequel, we will assume that the matrix recurrence coefficients diverge in a particular way: we will suppose that there exists a
sequence of positive definite matrices
$(D_n)_{n\in \mathbb N}
$
such that
\begin{gather}
\label{divergence_conditions}
\begin{cases}
\displaystyle \lim_{n\to \infty} D_n^{-1/2}A_n \, D_n^{-1/2} =A \, , \
 \lim_{n\to \infty} D_n^{-1/2}B_n \, D_n^{-1/2} =B \, , \\
\displaystyle \lim_{n\to \infty} D_n^{-1/2}C_n \, D_n^{-1/2} =C \, , \
 \lim_{n\to \infty} D_n^{-1/2} D_{n-1}^{1/2} =I_{N
 } \, .
\end{cases}
\end{gather}
When unbounded coefficients are considered in the scalar case (assuming the same hypothesis given by~\eqref{divergence_conditions}), the outer ratio asymptotic  is then obtained for the scaled polynomials $p_n(c_n z)$. However,  in the matrix case  there is a large range of possibilities to define the scaled matrix polynomial $P(Hx)$ (cf.~\cite{Duran2}). From now on we are going to work with two notions of scaled matrix polynomials depending on the 
kind of orthogonality (left or right) that we will deal with. In the case of left-orthogonality, the suitable definition of scaled matrix polynomials was introduced by 
Durán in~\cite{Duran2}.
\begin{defi}[cf. \cite{Duran2}]\label{Descalamiento}
Given the sequences of recurrence coefficients
$(A_n)_{n\in\mathbb{N}}
$,
$(B_n)_{n\in\mathbb{N}}
$,
and
$(C_n)_{n\in\mathbb{N}}
$,
we can define a sequence of matrix polynomials in an one matrix variable,~$\{ \pmb V_{n} \}_{n \in \mathbb N}$~as
\begin{gather*}
T \, \pmb V_{n} (T)= A_n \, \pmb V_{n+1}(T)+B_n \, \pmb V_{n} (T)+C_n \, \pmb V_{n-1} (T) \, ,
\end{gather*}
with initial conditions $\pmb V_{-1} (T)=\pmb 0_{N
}, \pmb V_{0} (T)=I_{N
}$. We define the \emph{left-scaled matrix polynomials} as $
V_n^{H}(x) =\pmb V_{n} (H \, x)$.
On the other hand, the natural definition in order to scale the right-orthogonal polynomials is the following one. Using the recurrence coefficients we can define another matrix polynomial sequence of one matrix variable, $\{ \pmb G_{n} \}_{n \in \mathbb N}$, as
\begin{gather*}
{\pmb G}_n(T) \, T= {\pmb G}_{n-1}(T) \, A_{n-1}+ {\pmb G}_n(T) \, B_n+ {\pmb G}_{n+1}(T) \, C_{n+1} \, ,
\end{gather*}
with initial conditions ${\pmb G}_{-1}(x)=\pmb 0_{N
}$ and ${\pmb G}_{0}(x)=I_{N
}$. Now, we define the \emph{right-scaled matrix polynomials} as $
G_n^{H} (x)={\pmb G}_{n} (H \, x)$.
\end{defi}
Notice that, in particular, for each non-negative integer $k$ the scaled polynomial sequences
\ $\{ V_n^{D_k}(z) \}_{n \in \mathbb N} 
$ \  and \ \
$\{ G_n^{D_k}(z)\}_{n\in \mathbb N}
$, \ are biorthogonal with respect to a certain varying matrix of measures~$W_k$.  We will say that
$\{ V_n^{D_k} \}_{n \in \mathbb N} $ and $\{ G_n^{D_k} \}_{n \in \mathbb N}$ are matrix biorthogonal polynomials with varying recurrence coefficients. In this way, our main result  Theorem~\ref{main result} will be a consequence of a more general theorem on outer ratio asymptotics cf. Theorem
~\ref{auxiliary_varying_coeff_theorem}.

In the sequel, we associate with three given matrices $A $, $C$,~$B$, where $A,C$, are nonsingular, the left-orthogonal Chebyshev matrix polynomials of second kind
$\{ U_n^{C,B,A}\}_{n \in \mathbb N}
$
which are defined~by the recurrence formula
\begin{gather}\label{Chebyshev_recurrence_relation}
x \, U_n^{C,B,A}(x)=C \, U_{n+1}^{C,B,A}(x)+B \, U_n^{C,B,A}(x)+A \, U_{n-1}^{C,B,A}(x) \, , \ \ n \in \mathbb N \, ,
\end{gather}
with initial conditions $U_0^{C,B,A}(x)=I_{N
}$ and $U_{-1}^{C,B,A}(x)=\pmb 0_{N
}$, as well as
the right-orthogonal Chebyshev matrix polynomial of second kind
$\{ T_{m}^{A,B,C}\}_{m\in\mathbb{N}}
$
given~by
\begin{gather}\label{right_chebyshev_matrix_polynomials}
x \, T_{n}^{A,B,C}(x)=T_{n+1}^{A,B,C}(x) \, A+T_{n}^{A,B,C}(x) \, B+T_{n-1}^{A,B,C}(x) \, C \, , \ \ n \in \mathbb N \, ,
\end{gather}
where $T_0^{A,B,C}(x)=I_{N
}$ and $T_{-1}^{A,B,C}(x)=\pmb 0_{N
}$.
We denote by $W_{C,B,A}$ the matrix weight for which the polynomial sequences $\{ U_n^{C,B,A} \}$, $\{ T_{n}^{A,B,C} \}$ are biorthogonal.

\section{Outer ratio asymptotics for orthogonal polynomials with var\-ying recurrence coefficients}  \label{sec:3}

For each $k=1,2,\ldots $, we consider orthogonal matrix polynomials
$\{ R_{n,k}\}_{n\in \mathbb N}
$
and
$\{S_{n,k} \}_{n\in \mathbb N}
$,
given by the recurrence relations
\begin{gather} \label{recurrence_relation_varying_coefficients}
x \, R_{n,k}(x) =A_{n,k} \, R_{n+1,k}(x) + B_{n,k} \, R_{n,k}(x) + C_{n,k} \, R_{n-1,k}(x) \, , \ \ n \in \mathbb N \, , \\
\label{recurrence_rel_ varying_right}
\phantom{o} \mbox{\hspace{-.35cm}} x \, S_{n,k}(x) =S_{n-1,k}(x) \, A_{n-1,k} +  S_{n,k}(x) \, B_{n,k} +  S_{n+1,k}(x) \, C_{n+1,k} \, , \ \ n \in \mathbb N \, ,
\end{gather}
with $R_{0,k}(x)=I_{N
}$, $R_{-1,k}(x)=\pmb 0_{N
}$, and $S_{0,k}(x)=I_{N
}$, $S_{-1,k}(x)=\pmb 0_{N
}$.

For a fixed~$k$, these matrix polynomial sequences are biorthogonal with respect to a certain quasidefinite matrix of measures which we denote by~$W_k $.

As far as we know, the only result on outer ratio asymptotics for matrix polynomials satisfying nonsymmetric recurrence relations is the following one.

\begin{teo}[cf. \cite{BMF1}, Theorem 3]
Let $W$ be a quasidefinite matrix of measures, and
$\{ V_n \}_{n\in \mathbb N}
$,
$\{ G_n \}_{n\in \mathbb N}
$
biorthogonal polynomial sequences with respect to $W$, satisfying,respectively, the three-term recurrence relation~\eqref{5trr_Vn},~\eqref{5trr_Gn}.
Let us assume 
\begin{gather*}
\lim_{n\to\infty}A_n=A \, , \
\lim_{n\to\infty}B_n=B \, , \
\lim_{n\to\infty}C_n=C \, ,
\end{gather*}
with $ A \, $, $C \, $ nonsingular matrices.
We denote by~$\Delta_n$  the set of zeros of $\{ V_n \}_{n\in \mathbb N}
$, and $\displaystyle \Gamma=\bigcap_{N \geq 0} M_N$, where $\displaystyle M_N=\overline{\bigcup_{n\geq N}\Delta_n}$.
Then,
\begin{gather*}
\lim_{n\to \infty}V_{n-1}(x) \, V_{n}^{-1}(x) \, A_{n-1}^{-1}=\int \frac{dW_{C,B,A}(y)}{x-y} \, , \ \  x \in \mathbb C \setminus \Gamma \, , \\
\lim_{n\to \infty} C_{n}^{-1} \, G_{n}^{-1} (x) \, G_{n-1} (x)=\int \frac{dW_{C,B,A}(y)}{x-y} \, , \ \  x \in \mathbb C \setminus \Gamma \, ,
\end{gather*}
where $W_{C,B,A}$ is the matrix of measures associated with the second kind Chebyshev matrix polynomials.
Moreover, the convergence is locally uniform on compact subsets~of~$\mathbb C \setminus \Gamma$.
\end{teo}
In the previous case, the coefficients in the recurrence relation are assumed to be convergent.

 The following result generalizes the previous one in two senses: we consider a case of varying recurrence coefficients and for a fixed~$k$, the recurrence coefficients will diverge in a particular way.

\begin{teo}\label{auxiliary_varying_coeff_theorem}
Let $W_k$ be for each $k\in \mathbb N$ a quasidefinite matrix of measures, and
$\{ R_{n,k} \}_{n\in \mathbb N}
$,
$\{ S_{n,k} \}_{n\in \mathbb N}
$
 be the sequences of biorthogonal matrix polynomials depending on a parameter
$k$, $k=1,2,\ldots$, satisfying~\eqref{recurrence_relation_varying_coefficients},~\eqref{recurrence_rel_ varying_right}.
Let $(n_m)_{m \in \mathbb N}$, $(k_m)_{m \in \mathbb N}$, be two increasing sequences of positive integers and we will assume that there exist three matrices $A$,~$B$, $C,$ with $A$ and~$B$ nonsingular, such that for all $l \in \mathbb N$, 
\begin{gather}
\label{convergent_subsequences_coefficients}
\lim_{m\to \infty}A_{n_m-l,k_m}=A \, , \
\lim_{m\to \infty}B_{n_m-l,k_m}=B \, , \
\lim_{m\to \infty}C_{n_m-l,k_m}=C \, .
\end{gather}
We denote by $\widetilde \Delta_{n,k}$ the set of zeros of $S_{n_m,k_m}$ and 
$\displaystyle \widetilde\Gamma=\bigcap_{N\geq 0} \widetilde M_{N,k}$, where 
$\displaystyle
\widetilde M_{N,k}=\overline{\bigcup_{n\geq N}\widetilde\Delta_{n,k}}$.
Then,
\begin{gather} \label{ratio_asymptotic}
\lim_{m\to \infty} R_{n_m-1,k_m}(x) \, R_{n_m,k_m}^{-1}(x) \, A_{n_m-1,k_m}^{-1} =\int \frac{dW_{C,B,A}(t)}{x-t} \, , \ \ x\in\mathbb{C}\setminus  \Gamma \, ,
  \\
\label{eq:ratioS}
\lim_{m\to \infty}
C_{n_m,k_m}^{-1} \, S_{n_m,k_m}^{-1}(x) \, S_{n_m-1,k_m}(x)  =\int \frac{dW_{C,B,A}(y)}{x-y} \, , \ \ x\in\mathbb{C}\setminus  \widetilde\Gamma \, ,
\end{gather}
where $W_{C,B,A}$ is the matrix weight for the generalized Chebyshev matrix polynomials defined in~\eqref{right_chebyshev_matrix_polynomials}.
Moreover, the convergence is locally uniform for $x$ on compact subsets of~$\mathbb{C}\setminus  \Gamma$.
\end{teo}
\begin{proof}
We will prove the asymptotic result~\eqref{eq:ratioS}. Notice that~\eqref{ratio_asymptotic} follows by using analogous arguments.
First, we consider the sequence of  discrete measures ~$\{\mu_{n,k} \}_{n\in\mathbb{N}}
$
defined~by
\begin{gather*}
\mu_{n,k}= \sum_{j=1}^{s} \delta_{x_{m,k,j}} R_{n-1,k}(x_{n,k,j}) \widetilde\Gamma_{n,k,j} S_{n-1,k}(x_{n,k,j}) \, , \ \ n \in \mathbb N \, ,
\end{gather*}
where $x_{n,k,j}$, $j=1,\ldots,s$ are the different zeros of the matrix polynomial, $R_{n,k} $, or, equivalently, the zeros of $S_{n,k}$ (cf. Theorem~\ref{zeros}) with 
multiplicities $\{\ell_ 1,\ldots, \ell_ s\}$,~and
\begin{gather*}
\widetilde \Gamma_{n,k,j}= S_{n-1,k}^{(1)}(x_{n,k,j})\frac{\ell_ j \big( \operatorname{Adj} \, (S_{n,k}(x))\big)^{(\ell_ j-1)}(x_{n,k,j})}{(\det(S_{n,k}(x)))^{(\ell_ j)}(x_{n,k,j})} \, .
\end{gather*}
Notice that from the definition of the quadrature formula	
\begin{multline*}
\int d\mu_{n,k} (x)
 =
\sum_{j=1}^{s}  R_{n-1,k}(x_{n,k,j}) \widetilde\Gamma_{n,k,j} S_{n-1,k}(x_{n,k,j}) \\
=\int R_{n-1,k} (x) \, dW_k(x) \, S_{n-1,k}(x)= I_{N
} \, , \ \ n = 1, 2, \ldots \, .
\end{multline*}
According to Lemma~\ref{decomposition}, we get
\begin{gather*}
\left( S_{n,k}(x) \right)^{-1} S_{n-1,k}(x)=  \sum_{k=1}^{s} \frac{D_{n,k,j}}{x-x_{n,k,j}} \, ,
\end{gather*}
where
\begin{gather}\label{coef_D_n,k}
D_{n,k,j}= \frac{\ell_ j}{(\det S_{n,k})^{(\ell_ j)} (x_{n,k,j})}  \big(\operatorname{Adj} \,  S_{n,k}(x)\big)^{(\ell_ j-1)} (x_{n,k,j}) S_{n-1,k}(x_{n,k,j}) \, .
\end{gather}
Multiplying in the left hand side of 
~\eqref{coef_D_n,k} by $C_{n,k}^{-1}$
\begin{gather*}
C_{n,k}^{-1} D_{n,k,j}= C_{n,k}^{-1} \frac{\ell_ j}{(\det S_{n,k})^{(\ell_ j)} (x_{n,k,j})}  \big(\operatorname{Adj} \,  S_{n,k}(x)\big)^{(\ell_ k-1)} (x_{n,k,j}) S_{n-1,k}(x_{n,k,j}) \, ,
\end{gather*} 
and applying the Liouville-Ostrogradski formula~\eqref{Liouville-Ostrogradski} 
\begin{gather*}
S_{n,k}(x_{n,k,j}) \big(\operatorname{Adj} \,  S_{n,k}\big)^{(\ell_ k-1)} (x_{n,k,j})= \pmb 0_{N
} \, ,
\end{gather*}
we get
\begin{gather*}
C_{n,k}^{-1} D_{n,k,j}= R_{n-1,k} (x_{n,k,j}) \, \widetilde\Gamma_{n,k} \, S_{n-1,k}(x_{n,k,j}) \, 
.
\end{gather*}
From the definition of the matrices $\widetilde \Gamma_{n,k} $, we have
\begin{gather*}
C_{n,k}^{-1} S_{n,k}^{-1}(x) \, S_{n-1,k}(x) =\int \frac{d\mu_{n,k}(y)}{x-y} \, , \ \ x\in \mathbb{C}\setminus  \widetilde\Gamma \, .
\end{gather*}
For two given nonnegative integers $n$, $k \, $
let us consider the generalized Chebyshev matrix polynomials of the second kind,
$\{T_n^{A,B,C}(x)\}_{n\in\mathbb{N}}
$,
defined in~\eqref{right_chebyshev_matrix_polynomials}. We can prove by induction that
\begin{gather}\label{limite_medida}
\lim_{m\to \infty} \int d\mu_{n_m,k_m}(x) T_l^{A,B,C}(x)=
I_{N
}\, \delta_{l,0} \, .
\end{gather}
To this end, we can write
\begin{gather}\label{desarrollo}
S_{n-1,k}(x) \, T_l^{A,B,C}(x) = S_{n,k}(x) \, K_{l,n-1,k}(x) + \sum_{i=1}^{n} S_{n-i,k}(x) \, \Delta_{i,l,n-1,k} \, ,
\end{gather}
where $K_{l,n-1,k}(x)$ is a matrix polynomial with degree less than or equal to $n-1$.~Thus,
\begin{multline*}
\int d\mu_{n,k}(x)T_l^{A,B,C}(x)=  \sum_{j=1}^{s}  R_{n-1,k} (x_{n,k,j}) \widetilde\Gamma_{n,k,j} S_{n-1,k}(x_{n,k,j})T_{l}^{A,B,C}(x_{n,k,j}) \\
 =  \sum_{j=1}^{s} R_{n-1,k}(x_{n,k,j})  \widetilde\Gamma_{n,k,j} \big( S_{n,k}(x_{n,k,j}) K_{l,n-1,k}(x_{n,k,j}) + \sum_{i=1}^{n} S_{n-i,k}(x_{n,k,j}) \Delta_{i,l,n-1,k} \big) \, .
\end{multline*}
According to the definition of the matrices $\widetilde \Gamma_{n,k}$ and taking into account 
that
\begin{multline*}
\big(\operatorname{Adj} \, (S_{n,k})(x)\big)^{(\ell_ k-1)} (x_{n,k,j}) \, S_{n,k}(x_{n,k,j}) \\
= S_{n,k}(x_{n,k,j}) \, \big(\operatorname{Adj} \, (S_{n,k})(x)\big)^{(\ell_ k-1)} (x_{n,k,j}) = \pmb 0_{N
} \, ,
\end{multline*}
we get
\begin{gather*}
\int d\mu_{n,k}(x) \, T_l^{A,B,C}(x)= \sum_{j=1}^{s} R_{n-1,k}(x_{n,k,j})  \, \widetilde\Gamma_{n,k,j} \, \big( \sum_{i=1}^{n} S_{n-i,k}(x_{n,k,j}) \, \Delta_{i,l,n-1,k} \big) \, .
\end{gather*}
Using the quadrature formula given in Theorem~\ref{quadrature_formula_right}, we~conclude
\begin{multline*}
\int d\mu_{n,k}(x) \, T_l^{A,B,C}(x)
 \\
= \sum_{i=1}^{n} \, \int  R_{n-1,k}(x) \, dW_k(x) \, S_{n-i,k}(x) \, \Delta_{i,l,n-1,k} = \Delta_{1,l,n-1,k} \, .
\end{multline*}
So,~\eqref{limite_medida} 
follows when 
$\displaystyle \lim_{m\to \infty} \Delta_{j,l,n_m-1,k_m}= I_{N 
} \, \delta_{j,l+1} \, $, holds.
We use induction on $l$. When $l=0$ the result is immediate. Now assuming that the result is valid up to $l$, the three-term recurrence relation for the matrix polynomials
$\{ T_{n}^{A,B,C}\}_{n\in \mathbb N}
$~yields
\begin{gather*}
S_{n-1,k} (x) \, T_{l+1}^{A,B,C}(x) = S_{n-1,k} (x) \, \big( x T_{l}^{A,B,C}(x)-T_{l}^{A,B,C}(x) \, B-T_{l-1}^{A,B,C}(x) \, C \big) \, A^{-1} \, .
\end{gather*}
Using~\eqref{desarrollo} and the three-term recurrence relation for
$\{S_{n,k}\}_{n\in\mathbb{N}}
$
\begin{multline*}
\Delta_{j,l+1,n-1,k}= A_{n-j,k} \, \Delta_{j-1,l,n-1,k} \, A^{-1} +   B_{n-j,k}\, \Delta_{j,l,n-1,k} \, A^{-1}  \\
+ C_{n-j} \, \Delta_{j+1,l,n-1,k} \, A^{-1}
- \Delta_{j,l,n-1,k} \, B \, A^{-1} - \Delta_{j,l-1,n-1,k} \, C \, A^{-1} \, .
\end{multline*}
For $j\geq l+3$ or $j\leq l-1$ the induction hypothesis shows that
\begin{gather*}
 \lim_{m\to  \infty} \Delta_{j,l+1,n_m-1,k_m}= \pmb 0_{N
 } \, .
\end{gather*}
We study the cases $j=l, j=l+1$, and $j=l+2$ separately:

{\noindent}{\bf Case 1}. $j=l$.
\begin{multline*}
\lim_{m\to  \infty} \Delta_{j,l+1,n_m-1,k_m}=\lim_{m\to \infty} \big( A_{n_m-l,k_m} \, \Delta_{l-1,l,n_m-1,k_m} \, A^{-1}  \\
+   B_{n_m-l,k_m} \, \Delta_{l,l,n_m-1,k_m} \, A^{-1} + C_{n_m-l}\, \Delta_{l+1,l,n_m-1,k_m} \, A^{-1} - \Delta_{l,l,n_m-1,k_m} \, B \, A^{-1}  \\
- \Delta_{l,l-1,n_m-1,k_m} \, C \, A^{-1}\big) =(C-C) \, A^{-1}=\pmb 0_{N
} \, .
\end{multline*}
{\bf Case 2}. $j=l+1$.
\begin{multline*}
\lim_{m\to  \infty} \Delta_{l+1,l+1,n_m-1,k_m}= \lim_{m\to \infty} \big( A_{n_m-l-1,k_m} \Delta_{l,l,n_m-1,k_m} \, A^{-1}  \\
+   B_{n_m-l-1,k_m}\Delta_{l+1,l,n_m-1,k_m} \, A^{-1} + C_{n_m-l-1}\Delta_{l+2,l,n_m-1,k_m} \, A^{-1} - \Delta_{l+1,l,n_m-1,k_m} \, B A^{-1} \\
  - \Delta_{l+1,l-1,n_m-1,k_m} \, C A^{-1}\big) =(B-B)A^{-1}= \pmb 0_{N
  } \, .
\end{multline*}
{\bf Case 3}. $j=l+2$.
\begin{multline*}
\lim_{m\to  \infty} \Delta_{l+2,l+1,n_m-1,k_m}=\lim_{m\to \infty} \big( A_{n_m-l-2,k_m} \Delta_{l+1,l,n_m-1,k_m} \, A^{-1}  \\
+   B_{n_m-l-2,k_m}\Delta_{l+2,l,n_m-1,k_m} \, A^{-1} + C_{n_m-l-2}\Delta_{l+3,l,n_m-1,k_m} \, A^{-1} - \Delta_{l+2,l,n_m-1,k_m} \, B A^{-1} \\
 - \Delta_{l+2,l-1,n_m-1,k_m} \, C A^{-1}\big) = A \, A^{-1}=I_{N
 } \, .
\end{multline*}
Now, in the same way that was done in~\cite{Duran2}, one can prove
\begin{gather*}
\lim_{m\to \infty} \int \frac{d\mu_{n_m,k_m}(y)}{x-y}=\int \frac{dW_{C,B,A}(y)}{x-y} \, , \ \ x\in\mathbb{C}\setminus  \Gamma \, ,
\end{gather*}
by using the so called \emph{method of moments}.
\end{proof}

\section{Main result} \label{sec:4}

\begin{defi}[cf. \cite{hor} Section 7.7]
Let $A,B \in \mathbb C^{N\times N} $ be Hermitian  matrices. We write $A\geq B$ if the matrix $A - B$ is positive semi-definite. Similarly, $A > B$ means that $A-B $ is positive definite.
\end{defi}
It is easy to see that the relation $\geq$ (respectively, $>$) is transitive and reflexive.
  \begin{defi}
 Let
$(A_n)_{n\in \mathbb N}
$ be a sequence of Hermitian matrices. We say that
$(A_n)_{n\in \mathbb N}
$
is an increasing sequence if $A_{n+1} \geq A_n$ for every $n\in \mathbb N$.
\end{defi}

\begin{teo} \label{lema:zeros}
Let the matrix sequence
$(D_n)_{n\in \mathbb N} $ be increasing. Then, for each $n \in \mathbb N $, the polynomials $V_n^{D_n}(x)$
and
$G_n^{D_n}(x)$
have the same zeros.
Moreover, denoting by $\widetilde x_{n,j}$ the zeros of
$V_n^{D_n}(x)$ or $G_n^{D_n}(x)$, then there exists a positive constant $M$ (independent of~$n$) such that the the zeros $\widetilde x_{n,j}$ are contained in a disk
$ \operatorname D=\{z\in \mathbb{C}: |z|<M\} \, $.
\end{teo}
\begin{proof}
The $N$-block Jacobi~matrix associated with $G_n^{D_k}(x) \, D_k^{1/2}$ is the transpose of the matrix
\begin{gather*}
\widetilde{J}^{(k)}=\left(\begin{matrix}
D_k^{-1/2} B_0D_k^{-1/2}  & D_k^{-1/2}  A_0D_k^{-1/2}  & \pmb 0_{N
} &   \\
D_k^{-1/2} C_1D_k^{-1/2}  & D_k^{-1/2} B_1D_k^{-1/2}  & D_k^{-1/2} A_1D_k^{-1/2}  & \ddots \\
 \pmb 0_{N
 } & D_k^{-1/2} C_2D_k^{-1/2} &D_k^{-1/2} B_2D_k^{-1/2}  & \ddots \\
  & \ddots & \ddots & \ddots
\end{matrix}\right) \, ,
\end{gather*}
associated to $D_k^{1/2} V_n^{D_k}(x)$.
The first statement is an straightforward consequence of the fact that the zeros of $V_n^{D_n} (x)$ are the eigenvalues of $\widetilde{J}^{(k)}_{nN}$ (truncated $N$-block Jacobi~matrix of dimension $nN$) and the zeros of $G_n^{D_n}(x) \, D_k^{1/2}$ are the eigenvalues of $\widetilde{J}_{nN}^{(k)T}$ (the transpose of the previous one).
Using the Gershgorin disk theorem for the location of eigenvalues, it is enough to show that the entries of the matrix~$\widetilde{J}_{nN}$ are bounded (independently of $n$). But the entries of this matrix~are
\begin{gather*}
D_k^{-1/2}  A_n \, D_k^{-1/2} = D_k^{-1/2} D_n^{1/2} D_n^{-1/2} A_n \, D_n^{-1/2} D_n^{1/2} D_k^{-1/2}, \\
D_k^{-1/2}  B_n \, D_k^{-1/2} = D_k^{-1/2} D_n^{1/2}D_n^{-1/2} B_n \, D_n^{-1/2} D_n^{1/2} D_k^{-1/2}, \\
D_k^{-1/2}  C_n \, D_k^{-1/2} = D_k^{-1/2} D_n^{1/2}D_n^{-1/2} C_n \, D_n^{-1/2} D_n^{1/2} D_k^{-1/2} \, .
\end{gather*}
Since  the matrices $ D_n^{-1/2} A_n \, D_n^{-1/2}$, $D_n^{-1/2} B_n \, D_n^{-1/2} ,$ and $ D_n^{-1/2} C_n \, D_n^{-1/2}$ converge and $I_{N
} \geq D_k^{-1/2} D_n^{1/2}$ if $n\leq k$ since
$(D_n)_{n\in \mathbb N}
$
is an increasing sequence, then the result follows.
\end{proof}

The theory of matrix orthogonal polynomials with varying recurrence coefficients studied in the previous section allows us to prove the main result of this manuscript, i.e. the outer ratio asymptotics for matrix biorthogonal polynomials satisfying three-term recurrence relations as in~\eqref{5trr_Vn} and~\eqref{5trr_Gn} with unbounded coefficients reads as follows.
\begin{teo}[Outer ratio asymptotics] \label{main result}
Let
$\{ V_n \}_{n\in \mathbb N}
$ and
$\{ G_n \}_{n\in \mathbb N}
$
be the sequences of biorthogonal matrix polynomials  with respect to a quasidefinite matrix of measures, $W$, satisfying the recurrence relations \eqref{5trr_Vn}, \eqref{5trr_Gn}, respectively. Let assume that there exists a sequence of positive definite matrices~$(D_n)_{n\in \mathbb N}
$
such that~\eqref{divergence_conditions} holds with $A$ and $C$ nonsingular matrices, and consider the scaled matrix polynomials sequences $\{ V_n^{D_n} \}$, $\{G_n^{D_n} \}$ given in Definition~\ref{Descalamiento}.
We denote by~$\Delta_n$ the set of zeros of $V_n^{D_n}$ and  by $\displaystyle  \Gamma=\bigcap_{N \geq 0} M_N$, where $\displaystyle M_N=\overline{\bigcup_{n\geq N}\Delta_n}$.
Then,

\hangindent=.75cm \hangafter=1
{\noindent}(a)
If we assume that the matrix sequence
$(D_n)_{n\in \mathbb N}
$
is increasing, then $\Gamma$ is a compact~set.

\hangindent=.75cm \hangafter=1
{\noindent}(b)
The following outer ratio asymptotics hold
\begin{gather*}
\lim_{n\to \infty} D_n^{1/2} V_{n-1}^{D_n}(z) \, \left(V_n^{D_n}(z)\right)^{-1} A_{n-1}^{-1} D_n^{1/2} =\int \frac{dW_{C,B,A}(t)}{z-t} \, , \ \ z \in\mathbb{C}\setminus  \Gamma, \\
\lim_{n\to \infty} D_n^{1/2} C_n^{-1} \left(G_{n}^{D_n}(z)\right)^{-1}G_{n-1}^{D_n}(z) \, D_n^{1/2}
=\int \frac{dW_{C,B,A}(t)}{z-t} \, , \ \ z
\in\mathbb{C}\setminus  \Gamma \, ,
\end{gather*}
where $W_{C,B,A}$ is the matrix weight for the generalized Chebyshev matrix polynomials of the second kind defined by~\eqref{Chebyshev_recurrence_relation}.
Moreover, the convergence is locally uniform for $z$ on compact subsets of~$\mathbb{C}\setminus  \Gamma$.
\end{teo}

\begin{proof}
Let
$(D_n)_{n\in \mathbb N}
$
be a sequence of $N\times N$ positive definite matrices. In order to apply Theorem
~\ref{auxiliary_varying_coeff_theorem}, we consider the scaled matrix polynomials $V_n^{D_k}(x)$, $G_n^{D_k}(x)$ associated with the parameters
$(A_n)_{n\in \mathbb N}
$,
$(B_n)_{n\in \mathbb N}
$,
$(C_n)_{n\in \mathbb N}
$.
Taking into account their definitions, we~have
\begin{gather}
x \, D_k  \,V_n^{D_k}(x) =A_n \, V_{n+1}^{D_k} \, (x)+B_n \, V_n^{D_k} (x)+C_n \, V_{n-1}^{D_k}(x) \, , \label{eq:scaledVn} 
\end{gather}
\begin{gather}
G_m^{D_k} \, (x) \, D_k \, x  = G_{m-1}^{D_k} (x) \, A_{m-1}+ G_m^{D_k} (x) \, B_m+ G_{m+1}^{D_k} (x) \, C_{m+1} \, , \label{eq:scaledGn}
\end{gather}
and so
\begin{multline*}
x D_k^{1/2} \, V_n^{D_k}(x)
=
D_k^{-1/2} A_n \, D_k^{-1/2} D_k^{1/2} \, V_{n+1}^{D_k}(x) \\
 +D_k^{-1/2}B_n \, D_k^{-1/2}D_k^{1/2} \, V_n^{D_k}(x)
+D_k^{-1/2}C_n \, D_k^{-1/2}D_k^{1/2} \, V_{n-1}^{D_k}(x) \, ,
\end{multline*}
\begin{multline*}
x  G_n^{D_k}(x) \, D_k^{1/2} = G_{n-1}^{D_k}(x) \, D_k^{1/2} D_k^{-1/2} A_{n-1}\, D_k^{-1/2} \\
+ G_n^{D_k}(x) \, D_k^{1/2} D_k^{-1/2} B_n \, D_k^{-1/2}
+ G_{n+1}^{D_k}(x) \, D_k^{1/2} D_k^{-1/2} C_{n+1} \, D_k^{-1/2} \, ,
\end{multline*}
respectively.
For each $k$, taking
\begin{gather*}
\left\{
\begin{array}{l}
R_{n,k}(x) = D_k^{1/2} V_n^{D_k}(x) \, , \
S_{n,k}(x) =  G_n^{D_k}(x) \, D_k^{1/2} \, , \\
A_{n,k} = D_k^{-1/2} A_n \, D_k^{-1/2} \, , \
B_{n,k} = D_k^{-1/2} B_n \, D_k^{-1/2} \, , \
C_{n,k} = D_k^{-1/2} C_n \, D_k^{-1/2} \, ,
\end{array}
\right.
\end{gather*}
the matrix polynomial sequences
$
\{ R_{n,k} \}_{n\in \mathbb N}
$,
$
\{ S_{n,k} \}_{n\in \mathbb N}
$
satisfy the 
three-term recurrence relations,~\eqref{recurrence_relation_varying_coefficients},~\eqref{recurrence_rel_ varying_right}, respectively,
 with initial conditions
\ $R_{0,k}(x)= G_{0,k}(x) =D_k^{1/2}$ \ and \ \ $R_{-1,k}(x)=G_{-1,k}(x)=\pmb 0_{N
}$.
Then, the matrix polynomial sequences
$ \{ R_{n,k} \}_{n\in \mathbb N}
$,~$\{ S_{n,k} \}_{n\in \mathbb N}
$
are orthogonal with respect to a certain varying matrix of measures~$W_k$.
Under the assumptions~\eqref{divergence_conditions} it is easy to see that the limit conditions~\eqref{convergent_subsequences_coefficients} are satisfied for $n_m=k_m=m$. Then, Theorem~\ref{main result} holds.

Finally, from Theorem~\ref{lema:zeros} we have that if the matrix sequence~$(D_n)_{n\in \mathbb N}
$
is increasing, then the zeros of $V_n^{D_n}(x)$
and
$G_n^{D_n}(x)$
are bounded (and so $\Gamma$ is a compact set), as we wanted to prove.
\end{proof}

\begin{example}
In Example~\ref{exa1} let us take the scalar measure $d\mu=x^{\alpha}e^{-x}dx,$ $\alpha>-1,$ supported on $ \, ]0,\infty[ \, $  and
$\{ L^\alpha_n \}_{n\in\mathbb N}
$
its respective monic orthogonal polynomial sequence (Laguerre polynomials of parameter $\alpha$). The biorthogonal sequences with respect to the new measure
$\left(\begin{smallmatrix}
x & -1 \\
0 & x
\end{smallmatrix} \right)
\, x^{\alpha}e^{-x}dx
$
are given by
\begin{gather*}
V_{n}(x)=\begin{pmatrix}
L_{n}^{\alpha+1}(x) & \frac{{L_{n}^{\alpha+1}(x)-
\frac{\alpha+n+1}{\alpha+1}L_{n}^{\alpha}(x)}}{x} \\[10pt]
0 & L_{n}^{\alpha+1}(x)
\end{pmatrix}=\begin{pmatrix}
L_{n}^{\alpha+1}(x)& -\frac{n}{\alpha+1}L^{\alpha+2}_{n-1}(x) \\[10pt]
&L_{n}^{\alpha+1}(x) \end{pmatrix} \, , \\[10pt]
 G_n(x)=-
 \begin{pmatrix}
 (-1)^n\frac{\Gamma(\alpha+1)K^{\alpha}_n(x,0)}{\Gamma(\alpha+n+1)} & 0\\[6pt]
 \frac{nK^{\alpha}_n(x,0)}{(\alpha+1)}+
 (-1)^n\frac{\Gamma(\alpha+1)}{\Gamma(\alpha+n+1)}\frac{\partial K^{\alpha}_n(x,0)}{\partial y} & (-1)^n\frac{\Gamma(\alpha+1)K^{\alpha}_n(x,0)}{\Gamma(\alpha+n+1)}
 \end{pmatrix} \, .
\end{gather*}
Moreover, the sequences
$\{ V_n \}_{n\in\mathbb N}
$,
$\{G_n \}_{n\in\mathbb N}
$
satisfy the 
three-term recurrence relations \eqref{5trr_Vn}, \eqref{5trr_Gn}
with $A_n = I_N,$ 
\begin{gather*}
B_n=\left(
\begin{matrix}
2n+\alpha+2 & -
{ 2}/{(\alpha+1)} \\ 
0 & 2n+\alpha+2\\
\end{matrix}\right) \, , \ \  C_n=\left(
\begin{matrix}
n(\alpha+n+1) & 0 \\ 
0 & n(\alpha+n+1)
\end{matrix}
\right) \, .
\end{gather*}
On the other hand, 
taking the sequence of positive definite matrices
$(D_n)_{n \in \mathbb N}
$, where
$ D_n=\begin{pmatrix}
n^2 & 0\\
0 & n^2
\end{pmatrix} \, $, \
we find that
$ D_n^{1/2}=\begin{pmatrix}
 n & 0 \\
 0 & n
 \end{pmatrix} \, $, \
$D_n^{-1/2}=\begin{pmatrix}
 1/n&0\\
 0&1/n
 \end{pmatrix} $ \
  and  \
$D_n^{-1/2} \, D_{n-1}^{1/2} = I_{N} \, $.
 Thus, \
 \begin{gather*}
 D_n^{-1/2}B_n \, D_n^{-1/2}=
 \left(
 \begin{matrix}
 {(2n+\alpha+2)}/{n^2} & - 
 {2}/({(\alpha+1) n^2}) \\
 0 & 
 {(2n+\alpha+2)}/{n^2} \\
 \end{matrix}
 \right) \, , \ \mbox{ and so  } \ \ D_n^{-1/2}B_n D_n^{-1/2} \to \pmb 0_N
 \, , \\
D_n^{-1/2}C_n D_n^{-1/2}=\left(
 \begin{matrix}
 \frac{\alpha+n+1}{n} &
 0 \\[6pt]
 0 & \frac{\alpha+n+1}{n}
 \end{matrix}
 \right) \, , \
\mbox{ and so } \ \
 D_n^{-1/2} C_n \, D_n^{-1/2}
      \to
 \begin{pmatrix}
 1 & 0 \\[6pt]
 0 & 1
 \end{pmatrix} \, .
 \end{gather*}
From Theorem~\ref{main result} we get
\begin{gather*}
\lim_{n \to \infty} D_n^{1/2} V_{n-1}^{D_n}(z)\left(V_n^{D_n}(z)\right)^{-1}  D_n^{1/2} =\int \frac{dW_{I,0,I}(t)}{z-t} \, ,
 \ \ z\in\mathbb{C}\setminus  \Gamma \, , \\
\lim_{n \to \infty} D_n^{1/2}C_n^{-1} \left(G_{n}^{D_n}(z)\right)^{-1}G_{n-1}^{D_n}(z) D_n^{1/2} =\int \frac{dW_{I,0,I}(t)}{z-t} \, , \ \  z\in\mathbb{C}\setminus  \Gamma \, ,
\end{gather*}
where $W_{I,0,I} $ is the matrix of measures 
associated with 
the sequence of orthogonal polynomials $\{ U_n^{C,0,I} \}_{n \in \mathbb N}$  satisfying
\begin{gather*}
x U_n^{I,0,I}(x)=U_{n+1}^{I,0,I}(x)+
U_{n-1}^{I,0,I}(x) \, , \ \ n\in \mathbb N \, , 
\end{gather*}
with initial conditions $U_0^{I,0,I}(x)=
I_{N
}$, 
$U_{-1}^{I,0,I}(x)=\pmb 0_{N
}$.

 Notice that it is an  interesting situation,  because  $W_{I,0,I}(x)$ is a positive definite matrix of measures.
Moreover, using Corollary~2.3 in~\cite{Duran1} we obtain
\begin{gather*}
\int \frac{dW_{I,0,I}(t)}{z-t}=\frac{z\, I}{2}-\frac{\sqrt{(z^2-4) \, I}}{2} \, .
\end{gather*}
\end{example}

\begin{example}
Let
$\{ V_n \}_{n\in \mathbb N}
$,
$\{ G_n \}_{n\in \mathbb N}
$
be matrix polynomial sequences satisfying the recurrence relations~\eqref{5trr_Vn} and~\eqref{5trr_Gn}, respectively, with
\begin{gather*}
A_n=\left(
\begin{matrix}
2 n^2 & 0 \\
7 n^2+1 & 5 n^2
\end{matrix}
\right) \, , \ \  B_n=\left(
\begin{matrix}
3/n & 4/n \\ 
n & 8 n^2
\end{matrix}
\right) \, , \ \ C_n=\begin{pmatrix}
2n^2&2n\\
0 &n
\end{pmatrix} \, .
\end{gather*}
If we take the  sequence of positive definite matrices,
$(D_n)_{n\in \mathbb N}
$,
\begin{gather*}
D_n=
\frac{n^8}{(n^2-1)^2}
\begin{pmatrix}
{1}/{n^2}+{1}/{n^4}
&
-{2}/{n^3}\\ 
-{2}/{n^3}&{1}/{n^2}+{1}/{n^4}
\end{pmatrix} \, ,
\end{gather*}
then
$D_n^{1/2}=\frac{n^4}{n^2-1}
\begin{pmatrix}
1/n&-1/n^2 \\
-1/n^2&1/n
\end{pmatrix}
\, $,
$D^{-1/2}_n=
\begin{pmatrix}
1/n&1/n^2 \\
1/n^2&1/n
\end{pmatrix}
$,
$D_n^{-1/2}D_{n-1}^{1/2}\to I_{N
} \, $.
From here
\begin{gather*}
D_n^{-1/2}A_n \, D_n^{-1/2}=\left(
\begin{matrix}
2+{1}/{n^3}+{5}/{n^2}+{7}/{n} & ({7 n^3+7 n^2+1})/{n^4} \\
7+{1}/{n^2}+{7}/{n} & 5+{1}/{n^3}+{2}/{n^2}+{7}/{n}
\end{matrix}
\right) \, ,
\end{gather*}
and so,
$D_n^{-1/2}A_nD_n^{-1/2}
\to
\begin{pmatrix}
2&0\\
7&5
\end{pmatrix}
\, $,
\begin{gather*}
D_n^{-1/2}B_nD_n^{-1/2}=\left(
\begin{matrix}
({9 n^2+3 n+4})/{n^4} & ({8 n^3+5 n+3})/{n^4} \\
({9 n^4+3 n+4})/{n^5} & ({8 n^5+n^3+4 n+3})/{n^5}
\end{matrix}
\right) \, ,
\end{gather*}
$ D_n^{-1/2}B_nD_n^{-1/2}\to \begin{pmatrix}
0&0 \\ 0&8
\end{pmatrix} \, $,	
\begin{gather*}
D_n^{-1/2}C_nD_n^{-1/2}=\left(
\begin{matrix}
2+{3}/{n^2} & {5}/{n} \\
({3 n^2+2})/{n^3} & 1+{4}/{n^2}
\end{matrix}
\right) \, , \ \
D_n^{-1/2}C_nD_n^{-1/2}\to\begin{pmatrix}
2&0\\
0&1
\end{pmatrix} \, .
\end{gather*}
Theorem~\ref{main result} yields
\begin{gather*}
\lim_{n\to \infty} D_n^{1/2} V_{n-1}^{D_n}(z) \, \left(V_n^{D_n}(z)\right)^{-1} A_{n-1}^{-1} D_n^{1/2} =\int \frac{dW_{C,B,A}(t)}{z-t} \, , \ \ z \in \mathbb{C} \setminus  \Gamma,
   \\
\lim_{n\to \infty} D_n^{1/2}C_n^{-1} \left( G_{n}^{D_n}(z)\right)^{-1} G_{n-1}^{D_n}(z) \, D_n^{1/2} =\int \frac{dW_{C,B,A}(t)}{z-t} \, , \ \ z\in\mathbb{C}\setminus  \Gamma \, ,
\end{gather*}
where $dW_{C,B,A}(x)$ is the matrix of measures of the sequence of biorthogonal polynomials satisfying for all $n \in \mathbb N \, $,
\begin{gather*}
x U_n^{C,B,A}(x)=\begin{pmatrix}
2 & 0 \\
0 & 1
\end{pmatrix} U_{n+1}^{C,B,A}(x)+\begin{pmatrix}
0 & 0 \\ 0 & 8
\end{pmatrix}U_n^{C,B,A}(x)+\begin{pmatrix}
2 & 0\\
7 & 5
\end{pmatrix}
U_{n-1}^{C,B,A}(x) \, ,\\
xT_{n}^{A,B,C}(x)=T_{n+1}^{A,B,C}(x)\begin{pmatrix}
2 & 0\\
7 & 5
\end{pmatrix}+T_{n}^{A,B,C}(x)\begin{pmatrix}
0 & 0 \\ 0 & 8
\end{pmatrix}+T_{n-1}^{A,B,C}(x)\begin{pmatrix}
2 & 0\\
0 & 1
\end{pmatrix} \, ,
\end{gather*}
with initial conditions \ \
$U_0^{C,B,A}(x)=T_0^{C,B,A}(x)=I_{N
}$, 
$U_{-1}^{C,B,A}(x)=T_{-1}^{C,B,A}(x)=\pmb 0_{N
} \, $.	
\end{example}

\section{The singular case} \label{sec:5}

In this section, we study the case when the limit matrices $A$ or $C$ are singular. In~\cite{Duran2}, for the case of symmetric recurrence coefficients the authors proved that the ratio asymptotic also exists in the singular case, although they cannot compute explicitly the degenerate positive definite matrix of measures appearing in the limit. A similar argument can be applied for obtaining the existence of outer ratio asymptotics for matrix polynomials satisfying recurrence relations with nonsymmetric coefficients.
\begin{teo}\label{main result singular}
Let
$\{ V_n \}_{n\in \mathbb N}
$,
$\{ G_n \}_{n\in \mathbb N}
$
be the sequences of biorthogonal matrix polynomials with respect to a quasidefinite matrix of measures, $W$, satisfying the recurrence relations~\eqref{5trr_Vn},~\eqref{5trr_Gn}. Let us  suppose that~$(D_n)_{n \in \mathbb N}$ is an increasing sequence of matrices. Under the hypotheses of Theorem~\ref{main result}, if we assume the limit matrix~$A$ to be singular, then there exists a matrix of measures~$\nu_1$, for~which
\begin{gather*}
\lim_{n\to \infty} D_n^{1/2} V_{n-1}^{D_n}(z) \, \left(V_n^{D_n}(z)\right)^{-1} A_{n-1}^{-1} D_n^{1/2}=\int \frac{d\nu_1(t)}{z-t} \, , \ \ z \in \mathbb{C} \setminus  \Gamma \, .
\end{gather*}
Moreover, if the matrix $C$ is singular, then there exists a matrix of measures~$\nu_2$, such~that
\begin{gather*}
\lim_{n\to \infty} D_n^{1/2}C_n^{-1} \left(G_{n}^{D_n} (z) \right)^{-1}(z) \, G_{n-1}^{D_n}(z) \, D_n^{1/2}=\int \frac{d\nu_2(t)}{z-t} \, , \ \  z \in\mathbb{C}\setminus  \Gamma \, .
\end{gather*}
Moreover, if we write
$F_{A}(z)= \int\frac{d\nu_1(t)}{z-t}$ and $F_{C}(z)= \int\frac{d\nu_2(t)}{z-t}$, with $z\notin \mbox{{\em supp\/}}(\nu)$, then these analytic matrix functions satisfy the matrix equation
\begin{gather}\label{matrix_equation}
C \, F(z) \, A \, F(z) + (B-zI) \, F(z) + I= \pmb 0_{N
} \, .
\end{gather}
\end{teo}
\begin{proof}
Following the technique given in Section~4 of~\cite{Duran2}, it is enough to reduce the result to the case of varying recurrence coefficients and to use the matrix polynomials $x^l \, I_{N
}$ instead of $U_n^{C,B,A}(x)$ and $T_n^{A,B,C}(x)$ in the proof 
of Theorem
~\ref{auxiliary_varying_coeff_theorem}.

In order to prove that the Hilbert transforms of the measures~$\nu_1$ and $\nu_2$ satisfy the matrix equation~\eqref{matrix_equation},
let us remind that the polynomials $V_n^{D_k}(x)$ and~$G_n^{D_k}(x)$ satisfy~\eqref{eq:scaledVn},~\eqref{eq:scaledGn}, i.e.
\begin{gather*}
x D_k^{1/2} V_n^{D_k}(x) =D_k^{-1/2}A_n \, V_{n+1}^{D_k}(x)
+D_k^{-1/2} B_n \, V_n^{D_k}(x)+D_k^{-1/2} C_n \, V_{n-1}^{D_k}(x) \, ,
 \\
\phantom{0}\mbox{\hspace{-.5cm}}x  G_n^{D_k}(x) D_k^{1/2}  = G_{n-1}^{D_k}(x) A_{n-1}  D_k^{-1/2}
+ G_n^{D_k}(x)  B_n  D_k^{-1/2}
+ G_{n+1}^{D_k}(x)  C_{n+1}  D_k^{-1/2}  .
\end{gather*}
Let us multiply
the  right hand side of the first one by $ \left(V_n^{D_n}\right)^{-1}(x) \, D_n^{-1/2}$,
and the left hand side of the second one
by $D_n^{-1/2} \left(G_n^{D_n}\right)^{-1}(x)$. Now, we put $k=n$ and take limit as $n$ tends to infinity.
Then, from Theorem~\ref{main result}
\begin{gather*}
 \lim_{n\to \infty} D_n^{-1/2} A_n \, V_{n+1}^{D_k}(x) \, \left(V_n^{D_n}(x)\right)^{-1} D_n^{-1/2} = F_{A}^{-1}, \\
 \lim_{n\to \infty} D_n^{-1/2} \left(G_n^{D_n}\right)^{-1}(x) \, G_{n+1}^{D_n}(x) \, C_{n+1} \, D_n^{-1/2} = F_{C}^{-1} \, ,
\end{gather*}
and the result follows.
\end{proof}




\begin{thebibliography}{99}

\bibitem{ACM} C. Álvarez-Fernández, G. Ariznabarreta, J. C. García-Ardila, M. Mañas, F. Marcellán, {\em Christoffel transformations for matrix orthogonal polynomials in the real line and the non-Abelian 2D Toda lattice hierarchy}, Internat. Math. Res. Notices,   {\bf 5} (2017), 1285--1341.
	
\bibitem{ACM1} C. Álvarez-Fernández, G. Ariznabarreta, J. C. García-Ardila, M. Mañas, F. Marcellán, {\em Transformation theory and Christoffel formulas for matrix biorthogonal polynomials on the real line},  arXiv:1605.04617v7 [math.CA].


\bibitem{Berg}  C. Berg, {\em The matrix moment problem},
A. Branquinho, A. Foulquié Moreno (ed.),
{\em Lecture notes on orthogonal polynomials}, Advances in the theory of special functions and orthogonal polynomials, Nova Science Publishers, Inc, New York, 2008. 1--57.

\bibitem{Bra-Cot-Fou}  A. Branquinho, L. Cotrim,  A. Foulquié Moreno, {\em Matrix interpretation of multiple orthogonality}, Numer. Algorithms {\bf 55},
(2010), 19--37.  	




\bibitem{BMF2} A. Branquinho, F. Marcellán, A. Mendes,
{\em Vector interpretation of the matrix orthogonality on the real line},
Acta Appl. Math. \textbf{112} (2010), 357--383.


\bibitem{BMF1} A. Branquinho, F. Marcellán, A. Mendes,
{\em Relative asymptotics for orthogonal matrix polynomials},
Linear Algebra Appl. \textbf{437} (2012), 1458--1481.



\bibitem{DPS} D. Damanik, A. Pushnitski, B. Simon, {\em The analytic theory of matrix orthogonal polynomials}, Surv. Approx. Theory \textbf{4} (2008), 1--85.



\bibitem{Duran1} A. J. Durán, {\em Ratio asymptotics for orthogonal matrix polynomials},
J. Approx. Theory \textbf{100} (1999), 304--344.

\bibitem{Duran3} A. J. Durán, {\em A generalization of Favard's theorem for polynomials satisfying a recurrence relation},
J. Approx. Theory \textbf{74} (1993), 83--109.

\bibitem{Duran4} A. J. Durán, {\em Markov theorem for orthogonal matrix polynomials},
Canad. J. Math \textbf{48} (1996), 1180--1195.


\bibitem{Duran2} A. J. Durán, E. Daneri-Vias, {\em Ratio asymptotics for orthogonal matrix polynomials with unbounded recurrence coefficients},
J. Approx. Theory \textbf{110} (2001), 1--17.

\bibitem{D} A. J. Durán, P. López-Rodríguez, {\em Orthogonal matrix polynomials: zeros and Blumenthal's theorem,} J. Approx. Theory {\bf 84} (1996), 96--118.

\bibitem{DW1} A. J. Durán and W. Van Assche, {\em Orthogonal matrix polynomials and higher order recurrence relations},
Linear Algebra Appl. \textbf{219} (1995), 261--280.

\bibitem{N1} P. G. Nevai, {\em Orthogonal polynomials},
Mem. Amer. Math. Soc. \textbf{213} (1979).

\bibitem{gar1} J. C. García-Ardila, L. E. Garza,  F. Marcellán, {\em A canonical  Geronimus transformation for matrix orthogonal polynomials}, Linear Multilinear Algebra. In press. doi:10.1080/03081087.2017.1299089

\bibitem{GW1} J. S. Geronimo and W. Van Assche, {\em Orthogonal polynomials with asymptotically periodic recurrence coefficients},
J. Approx. Theory \textbf{46} (1996), 251--283.

\bibitem{hor} R. A. Horn, C. R. Johnson, {\em Matrix Analysis}, Second Edition, Cambridge University Press, Cambridge, 2013.

\bibitem{mar}  H. O. Yakhlef, F.  Marcellán, {\em Relative asymptotics of matrix orthogonal polynomials for Uvarov Perturbations: the Degenerate Case},
Mediterr. J. Math. \textbf{13} (2016), 3135--3153.

\bibitem{W2} W. Van Assche, {\em Asymptotics for orthogonal polynomials},
Lecture Notes in Math., Vol. 1265, Springer-Verlag, Berlin/New York, 1987.

\bibitem{W1} W. Van Assche, {\em Asymptotic properties of orthogonal polynomials from their recurrence formula I},
J. Approx. Theory \textbf{44} (1985), 258--276.

\end{thebibliography}
\end{document}